\documentclass[11pt,leqno]{amsart}
\usepackage{epsfig}
\usepackage{amssymb}
\usepackage{amscd}
\usepackage[matrix,arrow]{xy}
\usepackage{graphicx}

\newtheorem{theo}{Theorem}[section]
\newtheorem{lemma}[theo]{Lemma}
\newtheorem{defi}[theo]{Definition}
\newtheorem{prop}[theo]{Proposition}
\newtheorem{conj}[theo]{Conjecture}

\numberwithin{equation}{section}

\def\bL{\mathbb{L}}

\def\C{\mathbb{C}}
\def\Z{\mathbb{Z}}

\def\bL{{\mathbf L}}

\def\pre-tr{\operatorname{pre-tr}}

\def\Hom{\operatorname{Hom}}

\newcommand{\cD}{{\mathcal D}}

\newcommand{\cA}{{\mathcal A}}
\newcommand{\cB}{{\mathcal B}}

\newcommand{\cC}{{\mathcal C}}

\newcommand{\Perf}{\operatorname{Perf}}

\newcommand{\im}{\operatorname{Im}}

\newcommand{\colim}{\operatorname{colim}}

\newcommand{\Vect}{{\rm Vect}}

\newcommand{\Ho}{\operatorname{Ho}}

\newcommand{\id}{\operatorname{id}}

\newcommand{\dgcat}{\operatorname{dgcat}}

\usepackage{epsf}
\usepackage{amscd}

\title[Generalized non-commutative degeneration conjecture]
{Generalized non-commutative degeneration conjecture}

\author{Alexander I. Efimov}
\address{Steklov Mathematical Institute of RAS, Gubkin 8, Moscow 119991, Russia}
\email{efimov@mccme.ru}
\thanks{This work is supported by the RSF under a grant 14-50-00005.}

\begin{document}

\begin{abstract}
In this paper we propose a generalization of the Kontsevich--Soibelman conjecture on the degeneration of Hochschild-to-cyclic spectral sequence for smooth and compact DG category. Our conjecture states identical vanishing of a certain map between bi-additive invariants of arbitrary small DG categories over a field of characteristic zero.

We show that this generalized conjecture follows from the Kontsevich--Soibelman conjecture and the so--called conjecture on smooth categorical compactification.
\end{abstract}

\maketitle

\tableofcontents


\section{Introduction}

It is well known (\cite{H}, \cite{GH}) that de Rham cohomology of a compact K\"ahler manifold $X$ carries the Hodge decompositions:
$$H^n(X,\C)=\bigoplus\limits_{p+q=n}H^{p,q}(X),$$
where $H^{p,q}(X)=H^q(X,\Omega^p_X).$ This implies the following theorem.

\begin{theo}\label{th:Hodge_to_de_Rham_intro}(\cite{De}) Let $Y$ be a smooth projective algebraic variety over a field $k$ of characteristic zero. Then the spectral sequence
$$E_2^{p,q}=H^q(Y,\Omega^p_Y)\Rightarrow H^{p+q}_{DR}(Y)$$
degenerates at the second sheet.\end{theo}

Here $H^{\bullet}_{DR}(Y)$ denotes the algebraic de Rham cohomology, which is defined as hypercohomology of the algebraic de Rham complex (in Zariski topology):
$$H^n_{DR}(Y)=H^n_{Zar}(Y,(\Omega_Y^{\bullet},d_{DR})).$$

Theorem \ref{th:Hodge_to_de_Rham_intro} was also proved algebraically by Deligne and Illusie \cite{DI}, using reduction to positive characteristic and the Cartier isomorphism.

Theorem \ref{th:Hodge_to_de_Rham_intro} can be reformulated as degeneration of the Hochschild-to-cyclic spectral sequence of the DG (differential graded) category $D^b_{coh}(X).$ Here we identify the triangulated category $D^b_{coh}(X)$ with its DG enhancement. Kontsevich and Soibelman conjectured that such degeneration takes place for any smooth and compact DG category (see Definition \ref{defi:smooth_proper}).

Hochschild homology $HH_{\bullet}(\cA)$ and cyclic homology $HC_{\bullet}(\cA)$ of a small DG category $\cA$ are defined in Section \ref{sec:HH_mixed}. We denote by $u$ a formal variable of cohomological degree $2.$

\begin{conj}\label{conj:Kontsevich_degeneration_intro}(\cite{KS}) Let $\cA$ be a smooth and compact DG category over a field $k$ of characteristic zero. Then the spectral sequence
$E_1=HH_{\bullet}(\cA)\otimes_k (k[u^{\pm 1}]/uk[u])\Rightarrow HC_{\bullet}(\cA)$ degenerates at the first sheet.\end{conj}

It follows from the paper \cite{Ke1} that Conjecture \ref{conj:Kontsevich_degeneration_intro} in the case $\cA=D^b_{coh}(X)$
is indeed equivalent to Theorem \ref{th:Hodge_to_de_Rham_intro} for the variety $X.$

This conjecture was proved by Kaledin for DG algebras concentrated in non-negative degrees \cite{Ka}, via the method of Delifne--Illusie.

We propose a certain generalization of Conjecture \ref{conj:Kontsevich_degeneration_intro}, which states identical vanishing of a certain map between bi-additive invariants of small DG categories. 

In Section \ref{sec:HH_mixed} (see \eqref{eq:long_exact_HC^-}) we define for a small DG category $\cA$ the boundary map  $$\delta:HH_n(\cA)\to HC^-_{n+1}(\cA),$$ where $HC^-_{\bullet}(\cA)$ denotes the negative cyclic homology. 
We put $K_n(\cA)=K_n(\Perf(\cA))$ for all $n\in\Z,$ where we consider Waldhausen K-theory \cite{W}. Recall that we have a functorial Chern character on K-theory with values in Hochschild homology \cite{CT}:
$$ch:K_n(\cA)\to HH_n(\cA),\quad n\in\Z.$$
This Chern character passes through $HC^-_{\bullet}(\cA)$ (see \cite{CT}), but we will not need this.

\begin{conj}\label{conj:generalized_degeneration_intro} Let $\cB$ and $\cC$ be small DG categories over a field $k$ of characteristic zero. We denote by $\varphi_n$ the following composition:
$$\varphi_n:K_n(\cB\otimes\cC)\stackrel{ch}{\to}(HH_{\bullet}(\cB)\otimes HH_{\bullet}(\cC))_n\stackrel{\id\otimes\delta}{\to} (HH_{\bullet}(\cB)\otimes HC^{-}_{\bullet}(\cC))_{n+1}.$$
Then $\varphi_n=0$ for $n\leq 0.$\end{conj}

It is not hard to check (see Proposition \ref{prop:generalized_implies_Kontsevich}) that Conjecture \ref{conj:generalized_degeneration_intro} implies Conjecture \ref{conj:Kontsevich_degeneration_intro}.

Recall a conjecture on smooth categorical compactification. Denote by $\Ho_M(\dgcat_k)$ the homotopy category of small DG categories over $k$ with respect to Morita model structure \cite{Tab}. The notion of homotopically finite DG category is defined in \cite{TV}, Definition 2.4.

\begin{conj}\label{conj:categorical_compactification_intro} For any homotopically finite DG category $\cA$ there exist a smooth and compact DG category $\tilde{\cA}$ and an object $E\in\tilde{\cA},$ such that $\cA\cong \tilde{\cA}/E$ in $\Ho_M(\dgcat_k).$\end{conj}

To explain why we call Conjecture \ref{conj:categorical_compactification_intro} a conjecture on smooth categorical compactification, we consider the following example. Let $X$ be a smooth algebraic variety over a field $k$ of characteristic zero. According to theorems of Nagata \cite{N} and Hironaka \cite{Hi1}, \cite{Hi2}, there exists an open embedding $X\hookrightarrow \bar{X},$ where $\bar{X}$ is a smooth and proper variety over $k.$ Then we have an equivalence $$D^b_{coh}(X)\cong D^b_{coh}(\bar{X})/D^b_{coh,\bar{X}\setminus X}(\bar{X}).$$
Therefore, putting $\tilde{\cA}=D^b_{coh}(\bar{X}),$ $\cA=D^b_{coh}(X),$ for any generator $E\in D^b_{coh,\bar{X}\setminus X}(\bar{X})$ we have a Morita equivalence $\cA\simeq\tilde{\cA}/E.$ Since $\tilde{\cA}$ is smooth and compact, we can treat the DG category $\tilde{\cA}$ as "smooth categorical compactification"\, of the DG category $\cA.$

The main result of this paper is the following theorem.

\begin{theo}\label{th:Kontsevich_implies_generalized_intro} Conjectures \ref{conj:Kontsevich_degeneration_intro} and \ref{conj:categorical_compactification_intro} imply Conjecture \ref{conj:generalized_degeneration_intro}.\end{theo}

The paper is organized as follows.

In Section \ref{sec:DG_cat} we recall some basic constructions related with DG categories and their derived categories.

Section \ref{sec:HH_mixed} is devoted to the mixed Hochschild complex and cyclic homology. It consists mainly of definitions.

In Section \ref{sec:generalized_degeneration} we first proved that Conjecture \ref{conj:generalized_degeneration_intro} implies Conjecture \ref{conj:Kontsevich_degeneration_intro} (Proposition \ref{prop:generalized_implies_Kontsevich}). Then we prove Theorem \ref{th:Kontsevich_implies_generalized_intro}.

\section{DG categories}
\label{sec:DG_cat}

We refer the reader to the papers \cite{Ke3}, \cite{Ke4} for a general introduction to DG categories and DG modules. DG quotients of DG categories are introduced in the paper \cite{Dr}. The notion of homotopically finite DG algebras and DG categories is introduced in the paper \cite{TV}.

Fix some basic field $k.$ All DG categories under consideration will be defined over $k.$ Moreover, in this and other sections we put 
$$-\otimes-:=-\otimes_k-,\quad \Hom(-,-):=\Hom_k(-,-).$$

All DG modules in this paper are right by default. For a DG category $\cA$ we denote by $\cA^{op}$ the opposite DG category. For a DG functor $F:\cA\to\cB$ we denote by $F^{op}:~\cA^{op}\to\cB^{op}$ the corresponding functor between the opposite DG categories. For a pair of DG functors $F_1:\cA_1\to\cB_1,$ $F_2:\cA_2\to\cB_2$ we denote by $F_1\otimes F_2:\cA_1\otimes\cA_2\to\cB_1\otimes\cB_2$ their tensor product.

For a small DG category $\cA$ we denote by $\text{Mod-}\cA$ the abelian category of DG $\cA$-modules. We denote by $D(\cA)$ the derived category of $\cA,$ which is obtained from $\text{Mod-}\cA$ by inverting quasi-isomorphisms. We denote by $\Perf(\cA)\subset D(\cA)$ the full subcategory of perfect complexes. It is known to coincide with the subcategory of compact objects $\Perf(\cA)=D(\cA)^c.$

For a DG functor $F:\cA\to\cB$ we denote by $F_*:\text{Mod-}\cB\to\text{Mod-}\cA$ the restriction of scalars functor (i.e. the functor of composition with $F$). Its left adjoint (extension of scalars) is denoted by $F^*:\text{Mod-}\cA\to\text{Mod-}\cB.$ We have derived functors $F_*:D(\cB)\to D(\cA),$ $\bL F^*:D(\cA)\to D(\cB).$ Recall the notion of Morita equivalence.

\begin{defi}\label{defi:Morita_equivalence} A DG functor $F:\cA\to\cB$ between small DG categories is called a Morita equivalence if the functor
$\bL F^*:D(\cA)\to D(\cB)$ is an equivalence.\end{defi}

If $\cA$ is a small DG category, then for any DG modules $M\in D(\cA),$ $N\in D(\cA^{op})$ we have the derived tensor product $M\stackrel{\bL}{\otimes}_{\cA}N\in D(k).$

For any small DG category $\cA$ we denote by $I_{\cA}\in\text{Mod-}(\cA\otimes\cA^{op})$ the diagonal bimodule, which is given by the formula $$I_{\cA}(X,Y)=\cA(X,Y).$$ We also consider $I_{\cA}$ as an object of the category $\text{Mod-}(\cA\otimes\cA^{op})^{op},$ via the obvious equivalence $(\cA\otimes\cA^{op})^{op}\cong \cA^{op}\otimes\cA\cong\cA\otimes\cA^{op}.$

\begin{prop}\label{prop:DG_localizations} 1) Let $\cA$ be a small DG category, and $\cB\subset\cA$ --- small DG subcategory. We denote by $\pi:\cA\to\cA/\cB$ the projection DG functor. Then the functor $\bL\pi^*:\Perf(\cA)\to\Perf(\cA/\cB)$ is a localization up to direct summands. That is, the natural functor $\Perf(\cA)/\ker(\bL\pi^*)\to\im(\bL\pi^*)$ is an equivalence, and the Karoubi completion of the subcategory $\im(\bL\pi^*)\subset\Perf(\cA/\cB)$ coincides with $\Perf(\cA/\cB).$ Here we denote by $\im(\bL\pi^*)$ the essential image of the functor $\bL\pi^*.$

2) Suppose that a DG functor $F:\cA\to\cA'$ induces a localization up to direct summands $\bL F^*:\Perf(\cA)\to\Perf(\cA').$ Then the functor $\bL F^*:D(\cA)\to D(\cA')$ is a localization.

3) Suppose that a DG functor $F:\cA\to\cA'$ induces a localization $\bL F^*:D(\cA)\to D(\cA').$ Then we have an isomorphism
$$\bL(F\otimes F^{op})^*(I_{\cA})\cong I_{\cA'}$$
in $D(\cA'\otimes\cA'^{op}).$\end{prop}

\begin{proof}1) follows from \cite{Dr}, Theorem 1.6.2. 

2) follows from \cite{E}, Proposition 3.7.

3) follows from \cite{E}, Proposition 3.5.
\end{proof}

Recall the notions of smoothness and compactness for DG categories.

\begin{defi}\label{defi:smooth_proper} A small DG category $\cA$ is called

1) smooth if $I_{\cA}\in\Perf(\cA\otimes\cA^{op});$

2) compact if for any two objects $X,Y\in Ob(\cA)$ we have $\cA(X,Y)\in\Perf(k).$\end{defi}

Recall the notion of (semi-orthogonal) gluing of DG categories.

\begin{defi}\label{defi:gluing} Let $\cA$ and $\cB$ be small DG categories, and $M\in\text{Mod-}(\cA\otimes\cB^{op})$ a DG bimodule. The gluing of $\cA$ and $\cB$  via the bimodule $M$ is a DG category which is denoted by $\cA\sqcup_{M}\cB,$ and which is defined as follows.

1) The objects are defined by the equality $Ob(\cA\sqcup_{M}\cB)=Ob(\cA)\sqcup Ob(\cB);$

2) The morphisms are defined by the formula
$$(\cA\sqcup_{M}\cB)(X,Y)=\begin{cases}\cA(X,Y) & \text{if }X,Y\in Ob(\cA);\\
                           \cB(X,Y) & \text{if }X,Y\in Ob(\cB);\\
                           M(X,Y) & \text{if }X\in Ob(\cA),\,Y\in Ob(\cB);\\
                           0 & \text{if }X\in Ob(\cB),\,Y\in Ob(\cA).
                          \end{cases}$$
                          
3) The composition in $\cA\sqcup_M\cB$ comes from the composition in $\cA$ and $\cB,$ and from the structure of $\cA\otimes\cB^{op}$-module on $M.$
\end{defi}

We will need the following facts about homotopically finite DG categories.

\begin{prop}\label{prop:properties_of_hfp} 1) Let $\cA$ and $\cB$ be small DG categories-which are Morita equivalent. If $\cA$ is homotopically finite, then $\cB$ is also homotopically finite.

2) Let $\cA$ be a smooth and compact DG category. Then $\cA$ is homotopically finite.

3) Let $\cA$ be a homotopically finite DG category. Then $\cA$ is smooth.

4) Let $\cA$ be a homotopically finite DG category, and $E\in Ob(\cA)$ an object. Then the DG quotient $\cA/E$ is also homotopically finite.

5) Let $\cA$ and $\cB$ be homotopically finite DG categories, and $M\in\Perf(\cA\otimes\cB^{op})$ a perfect bimodule. Then the gluing $\cA\sqcup_M\cB$ is also homotopically finite.\end{prop}

\begin{proof}The statements 1)-3) are proved in the paper \cite{TV}, Corollary 2.12, Corollary 2.13, Proposition 2.14. The statements 4)-5) are proved in the paper \cite{E}, Proposition 2.9, Proposition 4.9.\end{proof}

\section{Mixed Hochschild complex}
\label{sec:HH_mixed}

We identify homological and cohomological complexes in the standard way: if $V_{\bullet}$ is a homological complex, then the corresponding cohomological complex is given by $V^n=V_{-n}.$ The same identification takes place for graded vector spaces. The shift functor $[n]$ is always cohomological: $V[n]^m=V^{n+m}.$

Let $\cA$ be a small DG category. Its Hochschild homology is defined by the formula
$$HH_n(\cA):=H^{-n}(I_{\cA}\stackrel{\bL}{\otimes}_{\cA\otimes\cA^{op}}I_{\cA}).$$

It follows directly from the definition that Hochschild homology is multiplicative (K\"unneth formula, see \cite{L}, Section 4.3):
$$HH_{\bullet}(\cA\otimes\cB)\cong HH_{\bullet}(\cA)\otimes HH_{\bullet}(\cB).$$

Bar resolution of the diagonal bimodule gives the Hochschild chain complex which is denoted by $C_{\bullet}(\cA).$
 As a graded vector space, this complex is defined by the equality
$$C_{\bullet}(A)=\bigoplus\limits_{\substack{n\geq 0;\\X_0,\dots,X_n\in Ob(\cA)}}\cA(X_n,X_0)\otimes \cA(X_{n-1},X_n)[1]\otimes\dots\otimes \cA(X_0,X_1)[1].$$
The differential is the sum of $2$ components: $b=b_{\delta}+b_{\mu},$ where
$$b_{\delta}(a_n\otimes a_{n-1}\otimes\dots\otimes a_0)=\sum\limits_{i=0}^{n}\pm a_n\otimes a_{n-1}\otimes\dots d(a_i)\otimes\dots\otimes a_0,$$
$$b_{\mu}(a_n\otimes a_{n-1}\otimes\dots\otimes a_0)=\pm a_0a_n\otimes a_{n-1}\otimes\dots\otimes a_1+\sum\limits_{i=0}^{n-1}a_n\otimes\dots\otimes a_{i+1}a_i\otimes\dots\otimes a_0.$$

It will be convenient for us to use the reduced Hochschild complex, which we denote by $C_{\bullet}^{red}(\cA).$ As a graded vector space, this complex is given by equality
$$C_{\bullet}^{red}(A)=\bigoplus\limits_{\substack{n\geq 0;\\X_0,\dots,X_n\in Ob(\cA)}}\cA(X_n,X_0)\otimes\bar{\cA}(X_{n-1},X_n)[1]\otimes\dots\otimes \bar{\cA}(X_0,X_1)[1].$$
Here $$\bar{\cA}(X,Y)=\begin{cases}\cA(X,X)/k\cdot\id_X & \text{if }X=Y;\\
\cA(X,Y) & \text{otherwise.}
\end{cases}$$
It is easy to see that the differential $b$ on $C_{\bullet}(\cA)$ induces a well defined differential on $C_{\bullet}^{red}(\cA),$ which we also denote by $b.$ Moreover, the projection $C_{\bullet}(\cA)\to C_{\bullet}^{red}(\cA)$ is a quasi-isomorphism.

\begin{defi}A mixed complex is a triple $(K_{\bullet},b,B),$ where $K_{\bullet}$ is a graded vector space, $b:K_{\bullet}\to K_{\bullet}$ is a differential of homological degree $-1,$ $B:K_{\bullet}\to K_{\bullet}$ is a differential of homological degree $1,$ such that $bB+Bb=0.$\end{defi}

In other words, a mixed complex is a DG module over the DG algebra $k[B]/(B^2),$ where $\deg(B)=-1,$ $d(B)=0.$ A morphism of mixed complexes is said to be a quasi-isomorphism if it is a quasi-isomorphism of DG modules over $k[B]/(B^2).$

The complex $(C_{\bullet}^{red}(\cA),b)$ is equipped by an additional Connes-Tsygan differential \cite{Co}, \cite{T}, \cite{FT}. This differential is denoted by $B$ and is given by equality
$$B(a_n\otimes\dots\otimes a_0)=\sum\limits_{i=0}^n\pm \id_{X_{i+1}}\otimes a_i\otimes \otimes\dots\otimes a_0\otimes a_n\otimes\dots\otimes a_{i+1},$$
where $a_i\in\cA(X_i,X_{i+1}),$ $0\leq i\leq n,$ and we put $X_{n+1}:=X_0$ for convenience.

It is easy to check that $B^2=0$ and $bB+Bb=0.$ Hence, we have a mixed complex $(C_{\bullet}^{red}(\cA),b,B).$

Any DG functor $F:\cA\to\cB$ between small DG categories induces a morphism of mixed Hochschild complexes which we denote by
$$F^*:C_{\bullet}^{red}(\cA)\to C_{\bullet}^{red}(\cB).$$
We also denote by $F^*$ the induced map on Hochschild homology and negative cyclic homology (see definition below).

\begin{theo}(\cite{Ke2}) If a DG functor $F:\cA\to \cB$ is a Morita equivalence, then the induced map
$F^*:C_{\bullet}^{red}(\cA)\to C_{\bullet}^{red}(\cB)$ is a quasi-isomorphism of mixed complexes.\end{theo}

We will need the following observation.

\begin{lemma}\label{lem:additivity_of_Hochschild} Let $\cA,$ $\cB$ and $M$ be as in Definition \ref{defi:gluing}. Then we have a natural isomorphism of mixed complexes
\begin{equation}\label{eq:additivity_of_Hochschild} C_{\bullet}^{red}(\cA\sqcup_M\cB)\cong C_{\bullet}^{red}(\cA)\oplus C_{\bullet}^{red}(\cB).\end{equation}\end{lemma}

\begin{proof}Indeed, let us put $\cC:=\cA\sqcup_M\cB,$ and consider a sequence of objects $X_0,\dots,X_n\in~Ob(\cC),$ such that
$$\cC(X_n,X_0)\otimes\bar{\cC}(X_{n-1},X_n)\otimes\dots\otimes\bar{\cC}(X_0,X_1)\ne 0.$$
Note that $\cC(X,Y)=0$ for $X\in Ob(\cB),$ $Y\in Ob(\cA).$ It follows that either $X_0,\dots,X_n\in Ob(\cA),$ or $X_0,\dots,X_n\in Ob(\cB).$ This implies the isomorphism \eqref{eq:additivity_of_Hochschild}.\end{proof}

From now on, we denote by $u$ a formal variable of (cohomological) degree $2.$ For any graded vector space $K_{\bullet}$ we can construct a graded $k[u]$-module $$K_{\bullet}[[u]]:=\prod\limits_{n\geq 0}K_{\bullet}[-2n].$$
For any homogeneous endomorphism of the space $K_{\bullet}$ we denote by the same symbol the corresponding $k[u]$-linear homogeneous endomorphism of $K_{\bullet}[[u]].$

\begin{defi}\label{defi:cyclic_complexes} Let $\cA$ be a small DG category. 

1) The negative cyclic complex of $\cA$ is defined by the formula
$$CC_{\bullet}^{-,red}(\cA):=(C_{\bullet}^{red}(\cA)[[u]],b+uB).$$
Its homology is called negative cyclic homology, and is denoted by $HC^-_{\bullet}(\cA).$

2) The cyclic complex $\cA$ is defined by the formula $$CC_{\bullet}^{red}(\cA):=(C_{\bullet}^{red}(\cA)\otimes_k (k[u^{\pm 1}]/uk[u]),b+uB).$$
Its homology is called cyclic homology, and is denoted by $HC_{\bullet}(\cA).$\end{defi}

It follows directly from the definition that both $HC^{-}_{\bullet}(\cA)$ and $HC_{\bullet}(\cA)$ are $k[u]$-modules.

By the definition of negative cyclic complex, we have a short exact sequence of complexes
\begin{equation}
\label{eq:short_exact_negative_cyclic}
0\to CC^{-,red}_{\bullet}(\cA)[-2]\stackrel{u}{\to} CC^{-,red}_{\bullet}(\cA)\to C_{\bullet}^{red}(\cA)\to 0.\end{equation}
It gives a long exact sequence in homology, which is of the form 
\begin{equation}\label{eq:long_exact_HC^-}\dots\to HC^-_{n+2}(\cA)\stackrel{u}{\to} HC^-_n(\cA)\to HH_n(\cA)\stackrel{\delta}{\to} HC^-_{n+1}(\cA)\to\dots.\end{equation}
From now on we denote by $\delta$ the boundary map from \eqref{eq:long_exact_HC^-}.

\begin{lemma}\label{lem:obvious_commutative_diagrams} 1) Let $F:\cA\to \cB$ be a DG functor between small DG categories. Then we have commutative diagrams
\begin{equation}
\label{eq:comm_diagram_for_ch}
\begin{CD}
K_n(\cA) @> F^* >> K_n(\cB)\\
@V ch VV @V ch VV\\
HH_n(\cA) @> F^* >> HH_n(\cB),
\end{CD}
\end{equation}
and
\begin{equation}
\label{eq:comm_diagram_for_delta}
\begin{CD}
HH_n(\cA) @> F^* >> HH_n(\cB)\\
@V \delta VV @V \delta VV\\
HC^-_{n+1}(\cA) @> F^* >> HC^-_{n+1}(\cB).
\end{CD}
\end{equation}

2) Recall that for any small DG categories $\cC$ and $\cD$ we denote by $\varphi_n$ the composition 
$$\varphi_n=(\id\otimes\delta)\circ ch:K_n(\cC\otimes\cD)\to (HH_{\bullet}(\cC)\otimes HC^-_{\bullet}(\cD))_{n+1}.$$
Let $F_1:\cA_1\to\cB_1,$ $F_2:\cA_2\to\cB_2$ be DG functors between small DG categories. Then we have a commutative diagram
\begin{equation}
\label{eq:comm_diagram_for_varphi_n}
\begin{CD}
K_n(\cA_1\otimes\cA_2) @> (F_1\otimes F_2)^* >> K_n(\cB_1\otimes\cB_2)\\
@V \varphi_n VV @V \varphi_n VV\\
(HH_{\bullet}(\cA_1)\otimes HC^-_{\bullet}(\cA_2))_{n+1} @> F_1^*\otimes F_2^* >>   (HH_{\bullet}(\cB_1)\otimes HC^-_{\bullet}(\cB_2))_{n+1}.
\end{CD}
\end{equation}
\end{lemma}

\begin{proof} Commutativity of the diagram \eqref{eq:comm_diagram_for_ch} is exactly the functoriality of Chern character \cite{CT}. Commutativity of the diagram \eqref{eq:comm_diagram_for_delta} follows from the fact that DG functor $F:\cA\to\cB$ induces a morphism of short exact sequences of complexes \eqref{eq:short_exact_negative_cyclic} for $\cA$ and $\cB.$ Finally, commutativity of the diagram \eqref{eq:comm_diagram_for_varphi_n} immediately follows from the commutativity of \eqref{eq:comm_diagram_for_ch} and \eqref{eq:comm_diagram_for_delta}.\end{proof}

\section{Generalized degeneration conjecture}
\label{sec:generalized_degeneration}

From now on we assume that the basic field $k$ is of zharacteristic zero.

The conjecture of Kontsevich and Soibelman, which was formulated in the introducton (Conjecture \ref{conj:Kontsevich_degeneration_intro}), admits the following reformulation. 

\begin{conj}\label{conj:Kontsevich_degeneration} Let $\cA$ be a smooth and compact DG category. Then the boundary map $\delta:~HH_{\bullet}(\cA)\to HC^-_{\bullet+1}(\cA)$ vanishes.\end{conj}

\begin{prop}\label{prop:equivalence_of_degeneration_conjectures} Conjecture \ref{conj:Kontsevich_degeneration} is equivalent to Conjecture \ref{conj:Kontsevich_degeneration_intro}.\end{prop}

\begin{proof} Note that the differentials in the spectral sequence $$E_1=HH_{\bullet}(\cA)\otimes_k (k[u^{\pm 1}]/uk[u])\Rightarrow HC_{\bullet}(\cA)$$
are the same as differentials in the spectral sequence
$$E_1=HH_{\bullet}(\cA)[[u]]\Rightarrow HC^-_{\bullet}(\cA).$$
Further, degeneration of the latter spectral sequence is equivalent to the existence of a (non-canonical) isomorphism of $k[u]$-modules $$HH_{\bullet}(\cA)[[u]]\cong HC^-(\cA).$$
This in turn is equivalent to the existence of a $k$-linear section of the projection $HC^-_{\bullet}(\cA)\to HH_{\bullet}(\cA).$ From the long exact sequence \eqref{eq:long_exact_HC^-} we obtain that the existence of such section is equivalent to vanishing of the boundary map $\delta:HH_{\bullet}(\cA)\to HC^-_{\bullet+1}(\cA).$\end{proof}

Recall a well known statement about non-degeneracy of the Chern character of diagonal bimodule.

\begin{prop}\label{prop:nondeg_copairing}(\cite{Sh}) Let $\cA$ be a smooth and proper DG category. Then the element $$ch(I_{\cA})\in HH_{\bullet}(\cA^{op}\otimes\cA)\cong HH_{\bullet}(\cA^{op})\otimes HH_{\bullet}(\cA)$$ induces an isomorphism
$$HH_{\bullet}(\cA^{op})^{\vee}\stackrel{\sim}{\to} HH_{\bullet}(\cA).$$\end{prop}

\begin{prop}\label{prop:generalized_implies_Kontsevich} Conjecture \ref{conj:generalized_degeneration_intro} implies Conjecture \ref{conj:Kontsevich_degeneration}, hence also Conjecture \ref{conj:Kontsevich_degeneration_intro}.\end{prop}

\begin{proof}Indeed, let us put $\cB:=\cA^{op},$ $\cC:=\cA,$ and consider the class $[I_{\cA}]\in K_0(\cB\otimes\cC).$
Conjecture \ref{conj:generalized_degeneration_intro} states that $\varphi_0([I_{\cA}])=0.$ But the element $\varphi_0([I_{\cA}])\in (HH_{\bullet}(\cA^{op})\otimes HC^{-}_{\bullet}(\cA))_1$ induces the map
$$\delta\circ\beta:HH_{\bullet}(\cA^{op})^{\vee}\to HC^-_{\bullet}(\cA)[-1],$$
where $\beta:HH_{\bullet}(\cA^{op})^{\vee}\stackrel{\sim}{\to} HH_{\bullet}(\cA)$ is an isomorphism from Proposition
\ref{prop:nondeg_copairing}. Since $\delta\circ\beta=0,$ we have $\delta=0.$ This proves the proposition.\end{proof}

Before we prove the main result, we will prove one more lemma. Denote by $\Vect_{\infty}$ the category of at most countable-dimensional vector spaces over $k.$  

\begin{lemma}\label{lem:vanishing_for_Vect_infty} We have $HH_{\bullet}(\Vect_{\infty})=0.$ Moreover, for any small DG category $\cA$ we have $K_{\bullet}(\Vect_{\infty}\otimes\cA)=0,$ \end{lemma}

\begin{proof}Note that any $k$-linear endofunctor $F:\Vect_{\infty}\to\Vect_{\infty}$ acts on the invariants $K_{\bullet}(\Vect_{\infty}\otimes\cA)$ and $HH_{\bullet}(\Vect_{\infty}).$ Moreover, $k$-linear endofunctors of the category $\Vect_{\infty}$ form an additive category, and the direct sum of endofunctors gives the sum of maps on invariants.

Let $V$ be a countable-dimensional vector space. We have an endofunctor $V\otimes-:\Vect_{\infty}\to \Vect_{\infty}.$ Choosing any isomorphism $V\cong V\oplus k,$ we obtain an isomorphism $$(V\otimes-)\cong(V\otimes-)\oplus\id$$
in the category of endofunctors $\Vect_{\infty}.$ It follows that the identity functor acts by zero on our invariants, which implies their vanishing.\end{proof}

Now we prove the main result of this paper.

\begin{theo}\label{th:Kontsevich_implies_generalized} Conjectures \ref{conj:Kontsevich_degeneration_intro} and \ref{conj:categorical_compactification_intro} imply Conjecture \ref{conj:generalized_degeneration_intro}.\end{theo}

\begin{proof} According to Proposition \ref{prop:equivalence_of_degeneration_conjectures}, we can consider Conjecture \ref{conj:Kontsevich_degeneration} instead of Conjecture \ref{conj:Kontsevich_degeneration_intro}.

Recall that for small DG categories $\cB$ and $\cC$ we denote by $\varphi_n$ the composition
$$\varphi_n=(\id\otimes\delta)\circ ch:K_n(\cB\otimes\cC)\to (HH_{\bullet}(\cB)\otimes HC^-_{\bullet}(\cC))_{n+1}.$$

Consider the following intermediate statements.

(i) For any homotopically finite DG category $\cA,$ we have $\varphi_0([I_{\cA}])=0,$ where $$\varphi_0:K_0(\cA\otimes\cA^{op})~\to~(HH_{\bullet}(\cA)\otimes HC^-_{\bullet}(\cA^{op}))_1.$$
Recall that according to Proposition \ref{prop:properties_of_hfp} homotopy finiteness of $\cA$ implies smoothness, hence the class $[I_{\cA}]\in K_0(\cA\otimes\cA^{op})$ is well defined.

(ii) For any homotopically finite DG categories $\cB$ and $\cC,$ the map $$\varphi_0:K_0(\cB\otimes \cC)\to (HH_{\bullet}(\cB)\otimes HC^{-}_{\bullet}(\cC))_1$$ equals to $0.$

(iii) For any small DG categories $\cB$ and $\cC,$ the map $\varphi_0:~K_0(\cB~\otimes~ \cC)~\to~ (HH_{\bullet}(\cB)\otimes HC^{-}_{\bullet}(\cC))_1$ equals to $0.$

We will divide the proof into several steps.

{\noindent{\it Step 1.}} First we will show that Conjectures \ref{conj:Kontsevich_degeneration} and \ref{conj:categorical_compactification_intro} imply (i).

Let $\cA$ be a homotopically finite DG category. According to Conjecture \ref{conj:categorical_compactification_intro} and Proposition \ref{prop:DG_localizations}, there exist a smooth and proper DG category $\tilde{\cA}$ and a DG functor $F:\tilde{\cA}\to\cA,$ such that we have an isomorphism $\bL(F\otimes F^{op})^*I_{\tilde{\cA}}\cong I_{\cA}$ in $D(\cA\otimes\cA^{op}).$ From this and from Lemma \ref{lem:obvious_commutative_diagrams} we obtain $\varphi_0([I_{\cA}])=(F^*\otimes (F^{op})^*)(\varphi_0([I_{\tilde{\cA}}])).$ Conjecture \ref{conj:Kontsevich_degeneration} immediately implies that $\varphi_0([I_{\tilde{\cA}}])=0.$ Thus, $\varphi_0([I_{\cA}])=0.$ This proves (i).

{\noindent{\it Step 2.}} We show that (i) implies (ii). Let $\cB$ and $\cC$ be homotopically finite DG categories, and $M\in\Perf(\cB\otimes\cC)$ a perfect bimodule. Consider the gluing $\cD:=\cB\sqcup_M\cC^{op}.$ By Proposition \ref{prop:properties_of_hfp}, 5), DG category $\cD$ is also homotopically finite. Denote by $\iota_{\cB}:\cB\to \cD$ and $\iota_{\cC}:\cC^{op}\to\cD$ the tautological embedding DG functors.

By Lemma \ref{lem:additivity_of_Hochschild} we have decompositions of mixed complexes $$C_{\bullet}^{red}(\cD)=C_{\bullet}^{red}(\cB)\oplus C_{\bullet}^{red}(\cC^{op}),\quad C_{\bullet}^{red}(\cD^{op})=C_{\bullet}^{red}(\cB^{op})\oplus C_{\bullet}^{red}(\cC).$$

In particular, we have a decomposition of the tensor product of Hochschild homology: 
\begin{multline}\label{eq:direct_sum_decomposition} HH_{\bullet}(\cD)\otimes HH_{\bullet}(\cD^{op})\cong HH_{\bullet}(\cB)\otimes HH_{\bullet}(\cB^{op})\oplus HH_{\bullet}(\cC^{op})\otimes HH_{\bullet}(\cC)\\\oplus HH_{\bullet}(\cB)\otimes HH_{\bullet}(\cC)\oplus HH_{\bullet}(\cC^{op})\otimes HH_{\bullet}(\cB^{op}).\end{multline}

We have a distinguished triangle in $\Perf(\cD\otimes\cD^{op})$ (see e.g. \cite{LS}):
\begin{equation}\label{eq:diagonal_for_gluing}\bL(\iota_{\cB}\otimes\iota_{\cC}^{op})^*M\to\bL(\iota_{\cB}\otimes\iota_{\cB}^{op})^*I_{\cB}\oplus \bL(\iota_{\cC}\otimes\iota_{\cC}^{op})^*I_{\cC}\to I_{\cD}
\end{equation}

From \eqref{eq:diagonal_for_gluing} we immediately obtain that the element  $ch([I_{\cD}])\in (HH_{\bullet}(\cD)\otimes HH_{\bullet}(\cD^{op}))_0$ has components $$ch([I_{\cB}])\in (HH_{\bullet}(\cB)\otimes HH_{\bullet}(\cB^{op}))_0,\quad ch([I_{\cC}])\in (HH_{\bullet}(\cC^{op})\otimes HH_{\bullet}(\cC))_0,$$
$$-ch([M])\in (HH_{\bullet}(\cB)\otimes HH_{\bullet}(\cC))_0,\quad 0\in HH_{\bullet}(\cC^{op})\otimes HH_{\bullet}(\cB^{op}),$$
with respect to the decomposition \eqref{eq:direct_sum_decomposition}.
From (i) we have $\varphi_0([I_{\cD}])=(\id\otimes\delta)(ch([I_{\cD}]))=0.$ By commutativity of the diagram \eqref{eq:comm_diagram_for_delta} we obtain that $\varphi_0([M])=(\id\otimes\delta)(ch([M]))=0.$ This proves (ii).

{\noindent{\it Step 3.}} Let us show that (ii) implies (iii). Let $\cB$ and $\cC$ be arbitrary small DG categories. 
According to \cite{TV}, Proposition 2.2, wecan represent $\cB$ and $\cC$ as filtered colimits of homotopically finite DG categories: $\cB=\colim\limits_{I}\cB_i,$ $\cC=\colim\limits_{J}\cC_j.$
It follows from \cite{TV}, Lemma 2.10, that $K_0$ commutes with filtered colimits of DG categories. In particular, we have \begin{equation}\label{eq:colimit_of_K_0} K_0(\cB\otimes\cC)=\colim_{I\times J}K_0(\cB_i\otimes\cC_j).\end{equation}

Denote by $f_i:\cB_i\to \cB,$ $g_j:\cC_j\to\cC$ the natural DG functors.
Take an arbitrary class $\alpha\in K_0(\cB\otimes\cC).$ Then \eqref{eq:colimit_of_K_0} implies that there exist $i\in I,$ $j\in J,$ and $\gamma\in K_0(\cB_i\otimes\cC_j),$ such that $\alpha=(f_i\otimes g_j)^*(\gamma).$ From (ii) we have $\varphi_0(\gamma)=0.$ From this and from Lemma \ref{lem:obvious_commutative_diagrams} we obtain $\varphi_0(\alpha)=(f_i^*\otimes g_j^*)(\varphi_0(\gamma))=0.$ This proves (iii).

{\noindent{\it Step 4.}} Let us show that (iii) implies Conjecture \ref{conj:generalized_degeneration_intro}. Let $\cB$ and $\cC$ be small DG categories. We prove by induction on $n\geq 0$ that the map $$\varphi_{-n}:K_{-n}(\cB\otimes\cC)\to (HH_{\bullet}(\cB)\otimes HC^{-}_{\bullet}(\cC))_{-n+1}$$
equals $0.$

The base of induction for $n=0$ coincides with the statement (iii).

Suppose that the statement of induction is proved for $n=l\geq 0.$ Let us prove it for $n=l+1.$

We have a natural fully faithful DG functor $k\otimes-:\cB\to \Vect_{\infty}\otimes\cB.$ Let us take the DG quotient $\Sigma\cB:=(\Vect_{\infty}\otimes\cB)/\cB.$ We have a natural quasi-equivalence.$\Sigma(\cB\otimes\cC)\simeq(\Sigma\cB)\otimes\cC.$

By Lemma \ref{lem:vanishing_for_Vect_infty}, we have vanishing $K_{\bullet}(\Vect_{\infty}\otimes\cB\otimes\cC)=0,$ $HH_{\bullet}(\Vect_{\infty}\otimes\cB)=0.$ From the long exact sequence of K-groups and from Lemma \ref{lem:vanishing_for_Vect_infty} we obtain the natural isomorphism $K_{-l-1}(\cB\otimes\cC)\cong K_{-l}(\Sigma(\cB)\otimes\cC).$ Moreover, from the long exact sequence of Hochschild homology and from Lemma \ref{lem:vanishing_for_Vect_infty} we obtain the natural isomorphism $HH_{\bullet}(\Sigma\cB)\cong HH_{\bullet}(\cB)[1].$
These isomorphisms fit into the following commutative diagram:
$$
\begin{CD}
K_{-l-1}(\cB\otimes\cC) @>ch >> (HH_{\bullet}(\cB)\otimes HH_{\bullet}(\cC))_{-l-1} @>\id\otimes\delta >> (HH_{\bullet}(\cB)\otimes HC^-_{\bullet}(\cC))_{-l}\\
@V\sim VV @V\sim VV @V\sim VV\\
K_{-l}(\Sigma(\cB)\otimes\cC) @>ch >> (HH_{\bullet}(\Sigma \cB)\otimes HH_{\bullet}(\cC))_{-l} @>\id\otimes\delta >> (HH_{\bullet}(\Sigma\cB)\otimes HC^-_{\bullet}(\cC))_{-l+1}.
\end{CD}
$$
Indeed, commutativity of the right square is obvious, and commutativity of the left square follows from the fact that Chern character gives a morphism of long exact sequences of localization in K-theory and Hochschild homology \cite{Ke2}, \cite{CT}.

By the inductive hypothesis, the composition of the lower horizontal arrows equals to zero. Since the vertical arrows are isomorphisms, the composition of upper horizontal arrows equals to zero. This proves the inductive step.

Theorem is proved.
\end{proof}

\end{document}